\documentclass[12pt,reqno]{amsart}
\usepackage[dvipsnames]{xcolor}

\usepackage[letterpaper,top=2.5cm,bottom=2cm,left=3cm,right=3cm,marginparwidth=1.75cm]{geometry}

\usepackage[ruled,vlined]{algorithm2e}
\usepackage{protosem}

\providecommand{\keywords}[1]{\small \textbf{Keywords}: #1}
\usepackage{ marvosym }
\usepackage{enumitem}
\usepackage{graphicx,amsfonts,amsmath,tikz-cd,amsthm,listings,amssymb}
\usepackage[final, colorlinks=true, linkcolor=cyan, citecolor=orange, urlcolor=magenta]{hyperref}
\usepackage[capitalise, noabbrev, nameinlink]{cleveref}
\usepackage{comment}

\usepackage{times}
\usepackage{epsfig}
\usepackage{graphicx}
\usepackage{amsmath}
\usepackage{amssymb}
\usepackage{tabularx} 
\usepackage{mathtools} 
\usepackage{amsfonts}
\usepackage{amsthm}
\usepackage[utf8]{inputenc}

\definecolor{codedarkgreen}{RGB}{51, 133, 4}
\definecolor{codemaroon}{RGB}{133, 5, 63}
\definecolor{codeteal}{RGB}{0, 128, 96}
\lstdefinelanguage{Macaulay2}{
basicstyle=\small\ttfamily,
alsoletter=",
classoffset=1,
keywords={matrix,minors,gb,transpose,det,ideal,apply,subsets,ker,gens,fold,flatten,entries},
keywordstyle={\color{blue}},
classoffset=2,
morekeywords={from, to, list},
keywordstyle={\color{codemaroon}},
classoffset=3,
morekeywords={QQ},
keywordstyle={\color{codedarkgreen}},
classoffset=4,
morekeywords={MonomialOrder},
keywordstyle={\color{codeteal}},
xleftmargin=1.5cm,
xrightmargin=1em,
columns=fullflexible,
keepspaces=true,
stepnumber=1,
numbers=none,
captionpos=b,
showspaces=false,
frame=none
}

\theoremstyle{plain}

\newtheorem{theorem}{Theorem}[section] 
\newtheorem{proposition}[theorem]{Proposition}
\newtheorem{lemma}[theorem]{Lemma}

\theoremstyle{definition}
 
\theoremstyle{plain}
\newtheorem{example}[theorem]{Example}

\newcommand{\RR}{\mathbb{R}}

\newcommand{\CC}{\mathbb{C}}

\begin{document}

\title{Caustics by Refraction of Circles and Lines}

\author[Felix Rydell]{Felix Rydell}



\maketitle

\begin{center}
\small
Swedish Defence Research Agency, Kista, Stockholm, Sweden\\
\texttt{felixrydell@gmail.com}\\
\url{https://orcid.org/0000-0003-0300-8115}
\end{center}

\begin{abstract}
    This short note revisits the classical result that the complete caustic by refraction of a circle is the evolute of Cartesian ovals. We provide additional details to the statement and geometric proof of this fact, as presented in G. Salmon's 1879 book `Higher Plane Curves'. We observe that as the circle tends to a line, the Cartesian ovals collapse into an ellipse or a branch of a hyperbola. Further, we derive a general formula for caustics by refraction of circles using a computer algebra system, providing a modern computational perspective on this classical problem.
\end{abstract}

\section*{Introduction} The \textit{envelope} of a family of curves in the plane is a curve that is tangent to each curve of the family at some point, such that these points of tangency form the whole curve. As Bruce and Giblin write, ``These [curves] appear to cluster along another curve, which the eye immediately picks out $\ldots$ The new curve is called the envelope'' \cite[Chapter 5]{bruce1992curves}. This effect is particularly noticeable in \Cref{fig: Evolutes-Conic Sections}. Important examples of envelopes are evolutes. The \textit{evolute} of a plane curve is the envelope of the family of normal lines, called \textit{normals}, and in differential geometry it is the locus of centres of curvature \cite[page 89]{bruce1992curves}. \Cref{fig: Evolutes} shows this envelope for an ellipse and other plane curves. The study of evolutes has a long standing history in mathematics. From 200 BC by Apollonius \cite{toomer2012apollonius} to Huygens \cite{huygens1986christiaan} and Salmon \cite{salmon1873treatise} in the 19th century. Recently, Piene, Riener and Shapiro continued this tradition, by studying various numerical invariants of evolutes in the plane \cite{piene2021return}. Evolutes of algebraic curves coincide with the classical Euclidean distance discriminant, the set of data points for which the number of distinct smooth complex critical points for the nearest point problem differs from the Euclidean distance degree \cite[Section 7]{draisma2016euclidean}.  

\begin{figure}
    \centering
     \includegraphics[width = 0.825\textwidth]{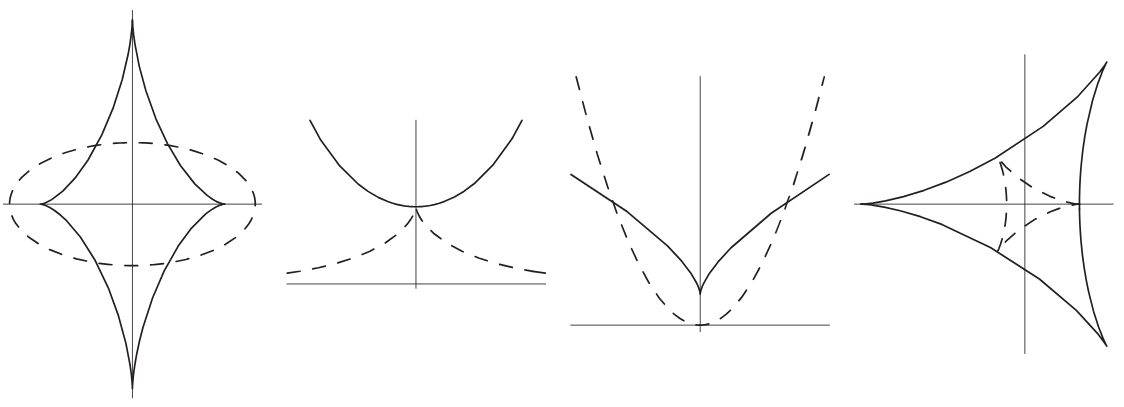}
    \caption{Plane curves (dashed) and their evolutes (solid). The image is taken from Mathworld--A Wolfram Web Resource   \cite{Evolute}.}
    \label{fig: Evolutes}
\end{figure}

In optics, a \textit{caustic} in the plane is an envelope of light rays which have been refracted by a curve $\mathcal C\subseteq \RR^2$. From the physics point of view, it is a curve of concentrated light. Caustics gained interest in the 17th century by Tschirnhaus, La Hire and Leibniz \cite{scarpello2005work}. In this paper, light rays are assumed to emanate from a fixed point $A$ in the affine plane or from a point at infinity (such as the sun), called the \textit{radiant} point. In the first case, we define for a smooth point $X\in \mathcal C$ the light ray $L(X)$ as the span of $A$ and $X$. In the second case, we define the light ray $L(X):=\{X+Y:Y\in L_A\}$, where $L_A$ is the line through the origin that meets $A$ at infinity. Any two lines that meet at a point $X$ make two angles. We say that the angle of a line $L_1$ with respect to a line $L_0$ meeting at $X$ is the unique angle $\theta\in (-\pi/2,\pi/2]$ that $L_0$ needs to be rotated around $X$ to become $L_1$. Denote the angle of $L(X)$ with respect to the normal line $N(X)$ at $X$ by $\sphericalangle_{\mathrm{in}}$. Fix a nonzero \textit{refraction constant} $n\in \RR$. If $\frac{1}{n}\sin \sphericalangle_{\mathrm{in}}\in [-1,1]$, then the \textit{refracted ray} $R_n(X)$ is the unique line whose angle $\sphericalangle_{\mathrm{out}}$ with respect to $N(X)$ at $X$ satisfies 
\begin{align}\label{eq: sin/sin}
    \frac{\sin \sphericalangle_{\mathrm{in}}}{\sin \sphericalangle_{\mathrm{out}}}=n.
\end{align}
In the case that $\sphericalangle_{\mathrm{in}}=0$, we set $\sphericalangle_{\mathrm{out}}=0$. The \textit{caustic by refraction} given $A,\mathcal C,$ and $n$ is the envelope of all refracted rays $R_n(X)$. The \textit{complete} caustic by refraction is the envelope of all refracted rays $R_{n}(X)$ and $R_{-n}(X)$. As we will see, from the algebraic point of view it is perhaps more natural to study complete caustics by refraction. We refer to Robert Ferréol's ``Encyclopedia of Remarkable Mathematical Shapes'' for a friendly survey on caustics. It is available on the web at \url{https://mathcurve.com/courbes2d.gb/caustic/caustic.htm}.

The aim of this paper is to survey the classical result and its proof that complete caustics by refraction of circles are evolutes of Cartesian ovals, as presented in Salmon's book \cite[pages 99-101]{salmon1873treatise}. We illustrate this fact in \Cref{fig: CausticFig}. This figure, as well as \Cref{fig: Evolutes-Conic Sections,fig: CausticFig line}, were created using \texttt{Mathematica} \cite{Mathematica}. In \Cref{s: Prelim}, we rewrite \eqref{eq: sin/sin} as an algebraic expression, discuss Cartesian ovals in some detail, and recall classical trigonometry. In \Cref{s: Circles}, we revisit Salmon's result in \Cref{thm: Cartesian oval}. We add numerous details to the statement and proof. We note that, starting with a Cartesian oval, it is possible to deduce a radiant point, a circle, and a refraction constant whose complete caustic is the evolute of the closure of that Cartesian oval. This is the content of \Cref{prop: reverse}. In \Cref{s: Line}, we show how these Cartesian ovals collapse into ellipses or branches of hyperbola as the circle we refract from degenerates into a line. Finally, we discuss computational approaches to determining complete caustics in \Cref{s: Comp}, and with the help of \texttt{Macaulay2} \cite{M2}, we find a general formula. We end by computing examples of complete caustics by refraction of a parabola, an ellipse and a hyperbola. 

\begin{figure}
    \centering
     \includegraphics[width = 0.55\textwidth, angle = 0]{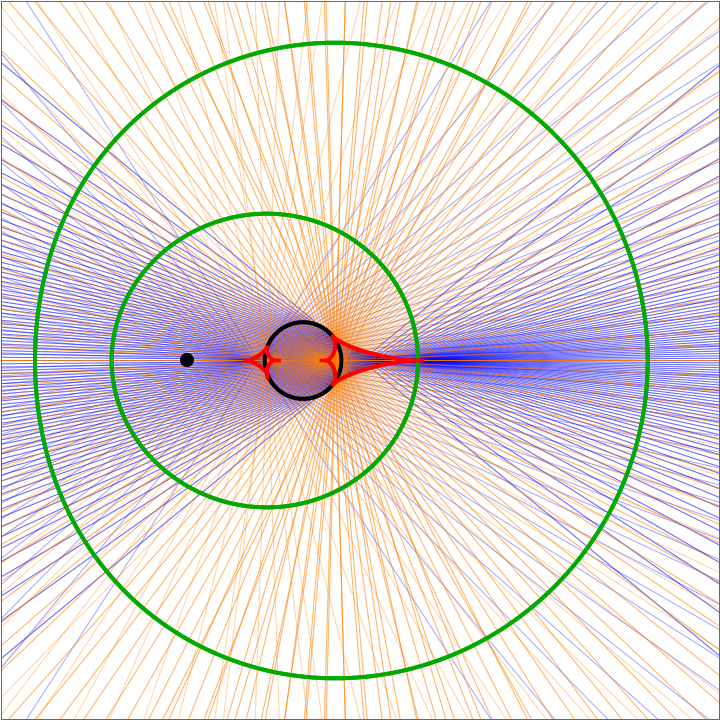}
    \caption{This illustration shows a black radiant point $A=(0,0)$ and a black circle $C$ of center $(1,0)$ and radius $r=1/3$. The blue lines are refracted rays given the refraction constant $n=1/2$ and the orange lines are refracted rays given $n=-1/2$. The red curve is the complete caustic by refraction. It is the evolute of the green Cartesian ovals, as described in \Cref{ex: main}. Some of the blue rays in this image are normals to the outer oval, and some are normals to the inner oval. Likewise for the orange rays. } 
    \label{fig: CausticFig}
\end{figure}

\bigskip

\paragraph{\textbf{Acknowledgements.}} This paper is dedicated to Vahid Shahverdi, a great friend and mathematician, who provided invaluable feedback for this project. The author would also like to thank Kathlén Kohn for helpful discussions and for introducing me to this subject. The author was supported by the Knut and Alice Wallenberg Foundation within their WASP (Wallenberg AI, Autonomous
Systems and Software Program) AI/Math initiative. 


\section{Preliminaries}\label{s: Prelim} \setcounter{equation}{0}
\numberwithin{equation}{section} In order to algebraically describe refracted rays, assume we are given a real radiant point $A\in \RR^2$, a real curve $\mathcal C\subseteq \RR^2$ that is the zero locus of the polynomial $G$, and a real refraction constant $n$. Let $U,V\in \RR^2$ be nonzero vectors. If $\theta\in (-\pi,\pi]$ is the angle that $U$ needs to be rotated to become parallel to $V$, then $U\cdot \iota(V)=\sin \theta \|U\|\|V\|$, where $\iota:V=(v_1,v_2)\mapsto (-v_2,v_1)$ rotates $V$ by $\pi/2$ radians, $U\cdot V$ is the standard inner product, and $\|U\|$ denotes the associated norm. For a smooth point $X\in \mathcal C$ away from $A$, recall the definition of $\sphericalangle_\mathrm{in}$ and $\sphericalangle_\mathrm{out}$ from the introduction. To determine the sines of these angles, we can use the above formula, although we have to adjust it slightly as we are dealing with lines and not just vectors. Let $\ell$ be the direction of the refracted ray, meaning that $R_n(X)$ is parametrized by $X+\ell a,a\in \RR$. Define $\sigma_X^{\mathrm{in}}\in \{-1,1\}$ as the sign of $(A-X)\cdot \nabla G(X)$ and $\sigma_X^{\mathrm{out}}\in \{-1,1\}$ as the sign of $\ell\cdot \nabla G(X)$. We then have 
\begin{align}
\sin\sphericalangle_{\mathrm{in}}=\sigma_X^{\mathrm{in}}\frac{(A-X)\cdot \iota\big(\nabla G(X)\big)}{\|A-X\|\|\nabla G(X)\|}\quad \textnormal{and}\quad \sin\sphericalangle_{\mathrm{out}}=\sigma_X^{\mathrm{out}}\frac{\ell\cdot \iota\big(\nabla G(X)\big)}{\|\ell\|\|\nabla G(X)\|}.
\end{align}
The assumed identity $\sin \sphericalangle_{\mathrm{in}}/\sin \sphericalangle_{\mathrm{out}}=n$ now yields the following algebraic expression
\begin{align}\label{eq: real point}
   \Big(\frac{\|\ell\|}{\|A-X\|}\Big)^2 \Big(\frac{(A-X)\cdot \iota(\nabla G(X))}{\ell\cdot\iota( \nabla G(X))}\Big)^2=n^2,
\end{align}
which does not care about the sign of $n$. In $\ell$, this is a homogeneous equation that is quadratic after clearing denominators. Then, for a fixed $X$, we expect two complex solutions in $\ell$ projectively. They correspond to the refraction constants $n$ and $-n$. From the algebraic perspective, we may take \eqref{eq: real point} as a definition for the refracted rays at $X$.

If $A$ is instead a point at infinity, then we simply replace $A-X$ in \eqref{eq: real point} by a fixed direction of the line $L_A$ from the introduction.



\subsection{Cartesian ovals}\label{ss: Cart o} Given two distinct \textit{foci} $A,B\in \RR^2$, and two real numbers $s$ and $t$, the set of real points $M$ satisfying
\begin{align}\label{eq: Cart st}
    \|A-M\|+s\|B-M\|=t
\end{align}
is a \textit{Cartesian oval}. It is contained in the algebraic degree-$4$ plane curve 
\begin{align}\label{eq: Cart alg st}
    \Big(\|A-M\|^2-s^2\|B-M\|^2+t^2\Big)^2-4t^2\|A-M\|^2=0
\end{align}
which we get by multiplying together the four expressions
\begin{align}\label{eq: four ovals}
   \|A-M\|+ \sigma_1 s\|B-M\|-\sigma_2 t,
\end{align}
where $\sigma_1,\sigma_2\in \{-1,1\}$. The history of the study of Cartesian ovals is outlined in \cite[Section 1.2]{breiding2024metric}. Generally, among the four ovals defined by \eqref{eq: four ovals}, two are real curves, although this is not always the case. In fact, if $s=1$ and $t>\|A-B\|$, then only one oval has real points and that oval is an ellipse. In general for $s=1$, \eqref{eq: Cart alg st} becomes quadratic, as we are left with
\begin{align}
    \Big(\|A\|^2-\|B\|^2+2M\cdot(B-A)+t^2\Big)^2-4t^2\|A-M\|^2=0.
\end{align}

To study the different types of Cartesian ovals of \eqref{eq: four ovals} for nonzero $s$ and $t$, assume without restriction that $s,t>0$. The following pairs of signs $(\sigma_1,\sigma_2)$ correspond to real curves.
\begin{enumerate}
\item For $s>1$:
       \begin{itemize}
        \item  If $t>\|A-B\|$:
        \begin{align}
            (\sigma_1,\sigma_2)=(1,1) \quad\textnormal{and}\quad (\sigma_1,\sigma_2)=(-1,-1).
        \end{align}
           \item  If $t<\|A-B\|$:
        \begin{align}
            (\sigma_1,\sigma_2)=(-1,1) \quad\textnormal{and}\quad (\sigma_1,\sigma_2)=(-1,-1).
        \end{align}
    \end{itemize}
            \item For $s=1$: 
        \begin{itemize}
        \item  If $t>\|A-B\|$:
        \begin{align}
            (\sigma_1,\sigma_2)=(1,1).
        \end{align}
           \item  If $t<\|A-B\|$:
        \begin{align}\label{eq: sig1sig2}
            (\sigma_1,\sigma_2)=(-1,1) \quad\textnormal{and}\quad (\sigma_1,\sigma_2)=(-1,-1).
        \end{align}
    \end{itemize}
    \item For $s<1$: 
        \begin{itemize}
        \item  If $t>s\|A-B\|$:
        \begin{align}
            (\sigma_1,\sigma_2)=(1,1) \quad\textnormal{and}\quad (\sigma_1,\sigma_2)=(-1,1).
        \end{align}
           \item  If $t<s\|A-B\|$:
        \begin{align}
            (\sigma_1,\sigma_2)=(-1,1) \quad\textnormal{and}\quad (\sigma_1,\sigma_2)=(-1,-1).
        \end{align}
    \end{itemize}
\end{enumerate}
In each case, the remaining choices of signs $(\sigma_1,\sigma_2)$ yield a Cartesian oval with no real points. The real curves represented by \eqref{eq: sig1sig2} are two branches of a hyperbola. All other Cartesian ovals in the list are bounded. To derive the list, we may consider the function 
\begin{align}
    f(M):= \|A-M\|+\sigma_1s\|B-M\|, \quad M\in \RR^2
\end{align}
and check which $\sigma_2 t$ lies in the interior of its image. For such $\sigma_2 t$, the preimage of $f$ is a real curve. As an example, given $\sigma_1=1$, the image of $f$ is the interval from the smallest number between $\|A-B\|$ and $s\|A-B\|$ to infinity. It is helpful to note that optima of $f$ lie on the line spanned by $A$ and $B$. For the special case $\sigma_1s=-1$, the function $f$ is bounded from both below and above as its range is $[-\|A-B\|,\|A-B\|]$.

If $s\neq 1$, the special cases $t=\|A-B\|$, respectively $t=s\|A-B\|$, yield the union of one real oval and one real point. When $\sigma_1s=1$ and $\sigma_2t= \|A-B\|$, the Cartesian oval is the line segment between $A$ and $B$.

We study the normals to Cartesian ovals below. First, given two distinct points $A$ and $B$, let $L_{AB}$ denote the line spanned by $A$ and $B$. Also, given three distinct points $A,B,$ and $C$, let $ABC$ denote the triangle they form, and $\sphericalangle_{BAC}\in [0,\pi]$ denote the inner angle of the triangle at $A$. 

\begin{lemma}\label{le: normals pre} Let $A$ and $B$ be distinct points, and let $s$ and $t$ be nonzero. The closure of the Cartesian oval
\begin{align}
\|A-M\|+s\|B-M\|=t
\end{align}
is smooth away from the line $L_{AB}$.
\end{lemma}

\begin{proof} For a point $M$ on the Cartesian oval to singular, the gradient of \eqref{eq: Cart alg st} must be zero. This condition gives us 
\begin{align}\label{eq: gamma}
  4\big((m_i-a_i)-s^2(m_i-b_i))\Gamma-8t^2(m_i-a_i) =0 \quad \textnormal{for}\quad  i=1,2,
\end{align}
where $\Gamma:= \|A-M\|^2-s^2\|B-M\|^2+t^2$. If $\Gamma=0$, then for the equations to hold, we must have $M=A$. If $\Gamma\neq 0$ and $(m_i-a_i)-s^2(m_i-b_i)=0$ for both $i$, then $m_i-a_i=0$ for \eqref{eq: gamma} to hold. Then, it must follow that $m_i-b_i=0$ for both $i$, which is a contradiction since it implies $M=A=B$. If $\Gamma\neq 0$ and $(m_i-a_i)-s^2(m_i-b_i)\neq 0$ for some $i$, then we can solve for $\Gamma$ and substitute it into the other equation. After simplification, we obtain
\begin{align}
    -(m_1-a_1)(m_2-b_2)=-(m_2-a_2)(m_1-b_1).
\end{align}
This identity is equivalent to the determinant of $2\times 2$ matrix $\begin{bmatrix}
    M-A & M-B
\end{bmatrix}$ being zero, which occurs if and only if $A,B,$ and $M$ are collinear. 
\end{proof}

\begin{lemma}\label{le: normals} Let $A$ and $B$ be distinct points, and let $s$ and $t$ be nonzero. Let $M$ be a fixed point on the Cartesian oval
\begin{align}\label{eq: Cart s}
\|A-M\|+\sigma_1s\|B-M\|=\sigma_2t
\end{align}
away from $L_{AB}$. \textnormal{(1)} The normal line $N(M)$ does not meet $A$ or $B$. \textnormal{(2)} Denote by $\alpha\in [0,\pi]$ an angle between $L_{AM}$ and $N(M)$ at $M$, and denote by $\beta\in [0,\pi]$ an angle between $L_{BM}$ and $N(M)$ at $M$. Then, \begin{align}\label{eq: abs}
    \frac{\sin \alpha}{\sin \beta}=s.
\end{align}
\textnormal{(3)} Apart from $N(M)$, there is exactly one more line $K$ through $M$ satisfying \eqref{eq: abs}. Exactly one line out of $N(M)$ and $K$ meets $L_{AB}$ inside the line segment between $A$ and $B$, and this line is $N(M)$ if and only if $\sigma_1=1$.    
\end{lemma}

\begin{proof} By \Cref{le: normals pre}, the point $M$ is smooth in the closure of the Cartesian oval. We show (1) at the end of the proof. We first prove (2). Let $G(M):=\|A-M\|+\sigma_1s\|B-M\|-\sigma_2t$. Then,
\begin{align}\label{eq: id sin}
    \sin \alpha=\frac{\big| (A-M)\cdot\iota\big(\nabla G(M)\big)\big|}{ \|A-M\|\|\nabla G(M)\|}\quad \textnormal{and}\quad
   \sin \beta =\frac{\big|(B-M)\cdot \iota\big(\nabla G(M)\big)\big|}{ \|B-M\|\|\nabla G(M)\|}.
\end{align}
Using that
\begin{align}\label{eq: nabla G}
   \nabla G(M)^\top=\begin{bmatrix}
     (m_1-a_1)/\|A-M\|+\sigma_1s(m_1-b_1)/\|B-M\|\\
         (m_2-a_2)/\|A-M\|+\sigma_1s(m_2-b_2)/\|B-M\|
    \end{bmatrix},
\end{align}
we get the desired equality \eqref{eq: abs} from \eqref{eq: id sin} and the fact that $M$ is away from $L_{AB}$. 

For (3), we study the two functions
\begin{align}\label{eq: frac}
    \frac{\sin (\sphericalangle_{AMB}-\gamma)}{\sin \gamma} \textnormal{ for }\gamma \in (0,\sphericalangle_{AMB}) \quad \textnormal{and}\quad  \frac{\sin (\gamma -\sphericalangle_{AMB})}{\sin \gamma}\textnormal{ for }\gamma \in (\sphericalangle_{AMB},\pi).
\end{align}
Differentiating with respect to $\gamma$, we see that both functions are injective. They are also surjective onto the positive real numbers. There are then two solutions in $\gamma$ yielding the ratio $s$, one in each interval of \eqref{eq: frac}, and one of them must correspond to the normal $N(M)$. 

We may assume by translation, rotation, and scaling that 
$A=(0,0)$ and $B=(1,0)$ without restriction. Denote by $I=\{(x_1,0):x_1\in (0,1)\}$ the interval between $A$ and $B$. For $M=(m_1,m_2)$ as in the statement, we now wish to find $\lambda$ such that $M+\lambda \nabla G(M)$ lies in $L_{AB}$, i.e., its last coordinate is $0$. Since we assumed $m_2\neq 0$, we get
\begin{align}
    \lambda = -\frac{1}{\frac{1}{\|A-M\|}+\sigma_1s\frac{1}{\|B-M\|}}.
\end{align}
Note that in the case that $\lambda$ is undefined with a denominator equal to $0$, then $\sigma_1 =-1$ and the normal does not intersect $I$. With $\lambda$ as above, the first coordinate of $M+\lambda \nabla G(M)$ is
\begin{align}
    m_1-\frac{1}{\frac{1}{\|A-M\|}+\sigma_1s\frac{1}{\|B-M\|}}\big( \frac{m_1}{\|A-M\|}+\sigma_1s\frac{m_1-1}{\|B-M\|}\big), 
\end{align}
which equals
\begin{align}\label{eq: sam}
    \frac{\sigma_1s\|A-M\|}{\|B-M\|+\sigma_1s\|A-M\|}.
\end{align}
The normal intersects $I$ if and only if \eqref{eq: sam} lies in the interval $(0,1)$. If $\sigma_1 =1$, then \eqref{eq: sam} lies in $(0,1)$ as $\|B-M\|+s\|A-M\|$ is always bigger than $s\|A-M\|$ and $\|B-M\|$. If $\sigma_1=-1$, then the values of \eqref{eq: sam} lie outside $[0,1]$ as is the case for the function $g(h):=-1/(h-1)=1/(1-h)$ defined over $h> 0$. 

Above, the normal $N(M)$ does not meet $(0,0)$ or $(1,0)$. This suffices to prove (1).
\end{proof}

\subsection{Classical trigonometry} Every closed arc in a circle corresponds to an angle $\theta\in [0,2\pi]$, the degree by which one end point is rotated around the center to become the other end point. The Inscribed Angle Theorem from Book 3 of Euclid's Elements says that if $A,B,$ and $C$ are distinct points inscribed in a circle, then the angle $\sphericalangle_{BAC}$ is half the angle of the closed arc defined by the endpoints $B$ and $C$ that does not contain $A$. As a direct consequence, the angle $\sphericalangle_{BXC}$ is constant for any $X$ in the same closed arc as $A$. Let $E$ be any point on the tangent of the circle at $C$ in the halfplane defined by $L_{BC}$ that is away from $A$. Then, by continuity, we must have $\sphericalangle_{BAC}=\sphericalangle_{BCE}$.

\begin{lemma}\label{cor: insc 2} Let $A,B,$ and $C$ be distinct point inscribed in a circle. If $E$ is the intersection of $L_{AB}$ and the tangent of the circle at $C$, then $ACE$ is similar to $BCE$, and
\begin{align}
    \frac{\|A-E\|}{\|A-C\|}=\frac{\|C-E\|}{\|C-B\|}.
\end{align}
\end{lemma}

A sketch of proof goes as follows. Both triangles $ACE$ and $BCE$ have by construction at least one angle in common, namely $\sphericalangle_{AEC}=\sphericalangle_{BEC}$. To see that they are similar it suffices to show that they have one more angle in common. This can be done by relating the inscribed angles to corresponding closed arcs and applying the Inscribed Angle Theorem. 

Another classical result from ancient Greece is Ptolemy's theorem. Let $A,B,C,$ and $D$ be vertices of a quadrilateral inscribed in a circle. Then,
\begin{align}\label{eq: Pto}
   \|B-C\| \|A-D\|+\sigma_1\|A-C\|\|B-D\|=\sigma_2\|A-B\|\|C-D\|,
\end{align}
with $(\sigma_1,\sigma_2)=(1,1)$ if $A$ is opposite to $B$ (and $C$ is opposite to $D$), $(\sigma_1,\sigma_2)=(-1,-1)$ if $A$ is opposite to $C$ (and $B$ is opposite to $D$), and $(\sigma_1,\sigma_2)=(-1,1)$ if $A$ is opposite to $D$ (and $B$ is opposite to $C$). Moreover, the converse is true. By this we mean that if $A,B,C,$ and $D$ satisfy \eqref{eq: Pto}, then these four points are inscribed in a circle. Finally, we recall Law of Sines, which dates back at least a thousand years. Let $A,B,$ and $C$ be distinct inscribed points on a circle. Then,
\begin{align}
    \frac{\|B-C\|}{\sphericalangle_{BAC}}=\frac{\|A-C\|}{\sphericalangle_{ABC}}=\frac{\|A-B\|}{\sphericalangle_{BCA}}=2r,
\end{align}
where $r$ is the radius of the circle. 


\section{Caustics by Refraction of Circles}\label{s: Circles} We are now ready to revisit the classical proof appearing in Salmon's book \cite[pages 99-101]{salmon1873treatise} that the complete caustic by refraction of circles are evolutes of Cartesian ovals, with details added. We first need a lemma that constitutes the basis for the proof of our main theorem.

\begin{lemma}\label{le: AOB} Let $\mathcal C$ be a circle with center $O$ and radius $r>0$. Assume that the point $A$ is away from $\mathcal C$ and $O$. For each point $R$ on $\mathcal C$, away from $L_{AO}$, define $\mathcal C(R)$ to be the unique circle through $A$ and $R$ that is tangent to $L_{RO}$ at $R$. Then, $\mathcal C(R)$ meets $L_{AO}$ at the point 
\begin{align}\label{eq: B def} B:=O+\frac{r^2}{\|A-O\|^2}(A-O),
\end{align}
independently of $R$. Further, $\|A-O\|\|B-O\|=r^2$.
\end{lemma}

The identity \eqref{eq: B def} is equivalent to
\begin{align}\label{eq: A B}
    A=O+\frac{r^2}{\|B-O\|^2}(B-O).
\end{align}
For the proof of \Cref{le: AOB}, we use that for three points $U,V,$ and $W$, we have
\begin{align}\begin{aligned}\label{eq: id id}
    \|V-W\|^2=& \|(V-U)-(W-U)\|^2\\  
    &\|V-U\|^2+\|W-U\|^2-2(V-U)\cdot (W-U).
\end{aligned}
\end{align}

\begin{proof} Let $E$ be the center of the circle $\mathcal C(R)$, and assume that $\mathcal C(R)$ is of radius $\gamma$. We show that the distance between $B$ and $E$ is exactly $\gamma$. By assumption, we have the following identities:
\begin{align}
\|R-O\|^2=r^2,\label{eq: r o}\\
    \|R-E\|^2=\gamma^2,\label{eq: r alph}\\
    \|A-E\|^2=\gamma^2,\label{eq: a alph}\\
    (R-O)\cdot (R-E)=0.\label{eq: norm}
\end{align}
We get by \eqref{eq: r o}, \eqref{eq: r alph}, and \eqref{eq: norm} that 
\begin{align}
    \|E-O\|^2=\|(R-O)-(R-E)\|^2=\gamma^2+r^2.
\end{align}
Next, for $B$ as defined in \eqref{eq: B def}, we have $\|B-O\|^2=r^4/\|A-O\|^2$ and $\|A-O\|\|B-O\|=r^2$. Moreover,\begin{align}\begin{aligned}
    2(B-O)\cdot (E-O)&=   \frac{r^2}{\|A-O\|^2} 2(A-O)\cdot (E-O)\\
    &=\frac{r^2}{\|A-O\|^2}(\|A-O\|^2+\|E-O\|^2-\|A-E\|^2)\\
    &= r^2+\frac{r^4}{\|A-O\|^2} .
\end{aligned}
\end{align}
Putting everything together, we obtain \begin{align}\label{eq: B-E}
    \|B-E\|^2&=\|B-O\|^2+\|E-O\|^2-2(B-O)\cdot(E-O)=\gamma^2,
\end{align}
which finishes the proof.
\end{proof}

\begin{theorem}\label{thm: Cartesian oval} Let $\mathcal C$ be a circle with center $O$ and radius $r>0$. Assume that the radiant point $A$ is away from $\mathcal C$ and $O$, and define $B$ as in \eqref{eq: B def}. Fix a real number $n\neq 0$. The set of refracted rays for the refraction constants $n$ and $-n$ that do not meet $A$ or $B$ coincides with the set of normals, excluding $L_{AO}$, to the real curves among the Cartesian ovals 
\begin{align}\label{eq: Cart}
  \|A-M\|+\sigma_1\frac{\|A-O\|}{r}\|B-M\|=\sigma_2\frac{\|A-O\|\|A-B\|}{r|n|},    \end{align}
  where $\sigma_1,\sigma_2\in \{-1,1\}$. In other words, the complete caustic is the evolute of the union of those real Cartesian ovals.
\end{theorem}

To determine which Cartesian ovals in \eqref{eq: Cart} are real curves, we refer to \Cref{ss: Cart o}. Observe that it is possible to consider radiant points on the circle $\mathcal C$ and get a complete caustic curve, however, the statement does not extend to that case.

  Since the proof below is very geometric in nature, we include the original illustration by Salmon in \Cref{fig: Salmon} to assist the reader. 

    \begin{figure}
        \centering
      \includegraphics[height=0.39\columnwidth]{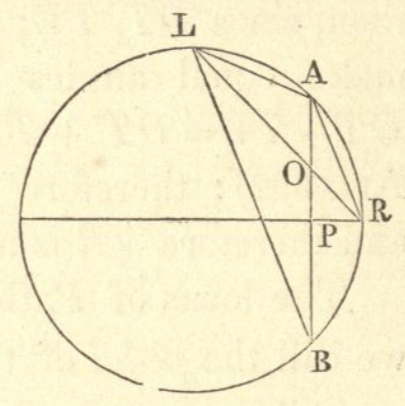}
        \includegraphics[height=0.39\columnwidth]{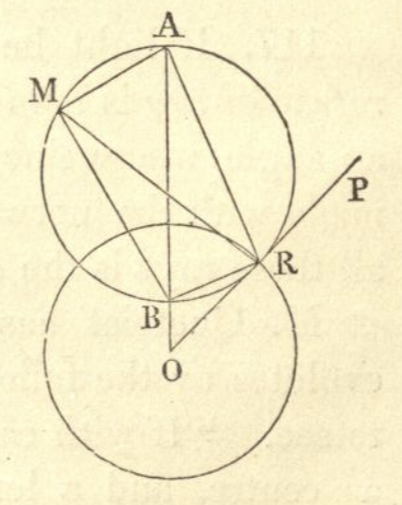}
    
        \caption{On the left: Salmon's illustration for the proof that the complete caustic by refraction of lines are the evolutes of ellipses. On the right: Salmon's illustration for the proof that the complete caustic by refraction of lines are the evolutes of Cartesian ovals. The points $A,B,R,O,$ and $M$ in the right figure have the same meaning in this illustration as in our proof of \Cref{thm: Cartesian oval}. These images are taken from \cite[page 100]{salmon1873treatise}.}
        \label{fig: Salmon}
    \end{figure}

\begin{proof} We start by taking a refracted ray that does not meet $A$ or $B$ and show that it is a normal to one of those Cartesian ovals. Take a point $R$ on the circle $\mathcal C$ away from $L_{AO}$ and consider the unique circle $\mathcal C(R)$ as in \Cref{le: AOB}. For $\sin \sphericalangle_{\mathrm{in}}= n\sin \sphericalangle_{\mathrm{out}}$ to hold, we must have that $\frac{1}{n}\sin \sphericalangle_{\mathrm{in}}\in[-1,1]$, which is the case at least for $R$ close enough to $L_{AO}$. A refracted ray meets the circle $\mathcal C(R)$ in two distinct points, $R$ itself and another point $M$. To see that they are indeed distinct, observe that otherwise the refracted ray would have to equal $L_{OR}$. This is only possible if $\sin \sphericalangle_{\mathrm{out}}= \sin \sphericalangle_{\mathrm{in}}=0$, but since $R$ is away from $L_{AO}$, this is not the case. Assume that the refracted ray $L_{RM}$ does not meet $A$ or $B$. We get a quadrilateral in $\mathcal C(R)$ given by the four distinct points $A,B,R,$ and $M$. Applying Ptolemy's theorem, we have
\begin{align}\label{eq: ptol}
 \|A-M\|+\sigma_1\frac{\|R-A\|}{ \|R-B\|}\|B-M\|=\sigma_2 \frac{\|R-M\|}{\|R-B\|}\|A-B\|,
\end{align}
with $(\sigma_1,\sigma_2)=(1,1)$ if $A$ is opposite to $B$ (and $R$ is opposite to $M$), $(\sigma_1,\sigma_2)=(-1,1)$ if $A$ is opposite to $M$ (and $B$ is opposite to $R$), and $(\sigma_1,\sigma_2)=(-1,-1)$ if $A$ is opposite to $R$ (and $B$ is opposite to $M$). By \Cref{cor: insc 2}, the triangles $AOR$ and $BOR$ are similar, and
\begin{align}\label{eq: sim}
    \frac{\|A-O\|}{\|A-R\|}=\frac{\|R-O\|}{\|R-B\|},
\end{align}
meaning that $\|R-A\|/\|R-B\|$ is fixed and equal to $\|A-O\|/r$. By the Law of Sines,
\begin{align}\label{eq: law of sines}
   \frac{\|R-A\|}{\sin \sphericalangle_{ABR}}=\frac{\|R-M\|}{\sin \sphericalangle_{MBR}}=\textnormal{ twice the radius of }\mathcal C(R).
\end{align}
From the Inscribed Angle Theorem it follows that $|\sin \sphericalangle_{\mathrm{in}}|=\sin \sphericalangle_{ABR}$ and $|\sin \sphericalangle_{\mathrm{out}}|=\sin \sphericalangle_{MBR}$. We conclude that $\|R-A\|/\|R-M\|$ equals $|n|$. By \eqref{eq: sim} and \eqref{eq: law of sines}, we get
\begin{align}\label{eq: AOcr}
    \frac{\|R-M\|}{\|R-B\|}=\frac{\|R-A\|}{\|R-B\|}\frac{\|R-M\|}{\|R-A\|}=\frac{\|A-O\|}{r|n|}.
\end{align}
Combining \eqref{eq: ptol},\eqref{eq: sim} and \eqref{eq: AOcr}, we get equation \eqref{eq: Cart}. We want to use \Cref{le: normals} to conclude that the refracted ray $L_{RM}$ is a normal to the Cartesian oval \eqref{eq: ptol}. By the Law of Sines and \eqref{eq: sim},
\begin{align}\label{eq: normal Cart}
    \frac{\sin \sphericalangle_{AMR}}{\sin\sphericalangle_{BMR}}=\frac{\|R-A\|}{\|R-B\|}=\|A-O\|/r.
\end{align}
Now, it suffices to note that $L_{RM}$ meets the line segment between $A$ and $B$ if and only if $A$ is opposite to $B$ in the quadrilateral, which occurs precisely when $\sigma_1=1$ as required in \Cref{le: normals}. 

Finally, we take with a point $M$ located away from $L_{AO}$ on one of the real curves among the Cartesian ovals in the statement. Let $G_1$ be the equation of that Cartesian oval, of the form \eqref{eq: Cart}. We show that $N(M)$ is a refracted ray and that it does not meet $A$ or $B$. The latter holds due to \Cref{le: normals}. Denote by $\mathcal C(M)$ the unique circle containing $A,B,$ and $M$. Since one of $A$ and $B$ is strictly inside $\mathcal C$ and the other is strictly outside, $\mathcal C(M)$ meets $\mathcal C$ at two points $R_1$ and $R_2$. As there is a unique circle containing $A,B,$ and one of $R_i$, and $\mathcal C(M)$ contains all these four points, we have by construction that $\mathcal C(R_1) =\mathcal C(R_2)=\mathcal C(M)$. We get two quadrilaterals defined by $A,B,M,$ and the fourth point $R_i$, for $i=1,2$. From these two quadrilaterals, we get two equations of the form \eqref{eq: ptol}. Exactly one of the quadrilaterals have that $R_i$ is opposite to $M$, as $R_1$ and $R_2$ are in different half-planes defined by $L_{AO}$. Therefore, the two equations have different signs $\sigma_1$. Let $R$ be the point among $R_1$ and $R_2$ whose corresponding equation has sign $\sigma_1$ equal to that of $G_1$, and write $G_2$ for this new equation involving $R$. By \eqref{eq: sim}, the left-hand sides of $G_1$ and $G_2$ are the same. Then, the right-hand sides must be the same and \eqref{eq: AOcr} must hold. We get that $\|R-A\|/\|R-M\|$ equals $|n|$ and by \eqref{eq: law of sines}, $\sin \sphericalangle_{ABR}/\sin  \sphericalangle_{MBR}$ is also $|n|$. Letting $L_{AR}$ be our light ray, we have $|\sin \sphericalangle_{\mathrm{in}}|=\sin \sphericalangle_{ABR}$ by the Insribed Angle Theorem. Let $\sphericalangle$ be the angle that $L_{RM}$ with respect to $L_{OR}$ at $R$. We similarly have $ |\sin \sphericalangle|=\sin  \sphericalangle_{MBR}$, and $|\sin \sphericalangle_{\mathrm{in}}|/|\sin \sphericalangle|=|n|$. This precisely means that $L_{RM}$ is a refracted ray for $n$ or $-n$. 
\end{proof}

\begin{example}\label{ex: main} Here, we explain \Cref{fig: CausticFig} in more detail. Set $A=(0,0),O=(1,0),r=1/2$ and $n=1/2$. Then, \eqref{eq: Cart} and \eqref{eq: Cart alg st} together yield the degree-4 equation
\begin{align}
    \Big(-72(y_1^2+y_2^2)+144y_1+192\Big)^2-9216(y_1^2+y_2^2)=0,
\end{align}
whose real part is the union of the two green ovals in \Cref{fig: CausticFig}. 
    \hfill$\diamondsuit\,$ 
\end{example}


As a consequence of \Cref{thm: Cartesian oval}, it is possible to start with a Cartesian oval and construct a radiant point and a circle such that the complete caustic curve is the evolute of the closure of that Cartesian oval. This is explored next.

\begin{proposition}\label{prop: reverse} Let $A$ and $B$ be distinct points. For $s,t>0$ with $s\neq 1$, consider the Cartesian ovals \begin{align}\label{eq: ex oval}\|A-M\|+\sigma_1s\|B-M\|=\sigma_2t,
\end{align}
where $\sigma_1,\sigma_2\in \{-1,1\}$. The evolute of the closure of the Cartesian ovals is the complete caustic by refraction given the radiant point $A$, the circle $\mathcal C$ of center 
\begin{align}
    O=A+\frac{s^2}{s^2-1}(B-A),
\end{align}
and radius $r=s\|A-B\|/|s^2-1|$, and the refraction constant $n=s\|A-B\|/t$.
\end{proposition}

\begin{proof} By \Cref{thm: Cartesian oval}, it suffices to find a circle $\mathcal C$ of center $O$ with radius $r>0$, and refraction constant $n>0$ such that
\begin{align}\begin{aligned}\label{eq: two eq}
   \frac{\|A-O\|}{r}=s\quad \textnormal{and}\quad   \frac{\|A-O\|\|A-B\|}{rn}=t.
\end{aligned}
\end{align}
From \eqref{eq: two eq}, we must have that $n=s\|A-B\|/t$. As in \Cref{le: AOB}, we assume $A,B,$ and $O$ are collinear; $O=A+q(B-A)$ for some $q\in \RR$. By \eqref{eq: two eq} and the definition of $B$, we have $B=O+(1/s^2)(A-O)$. Putting this together, we derive that $O-A=q(1-(1/s^2))(O-A)$, and we therefore set $q=s^2/(s^2-1)$. For this choice of $q$, $\|A-O\|/s$ equals $s\|A-B\|/|s^2-1|$, which we define $r$ as. As these choices of $O,r,$ and $n$ satisfy \eqref{eq: two eq}, we are done. 
\end{proof}


\section{Caustics by Refraction of Lines}\label{s: Line} Also found in \cite[pages 99-101]{salmon1873treatise} is a proof that the caustics by refraction of lines are normals to ellipses. We have illustrated this in \Cref{fig: CausticFig line}. We clarify here that this only occurs when $|n|<1$, otherwise the complete caustic is the evolute of a hyperbola. 

\begin{figure}
    \centering
     \includegraphics[width = 0.55\textwidth, angle = 0]{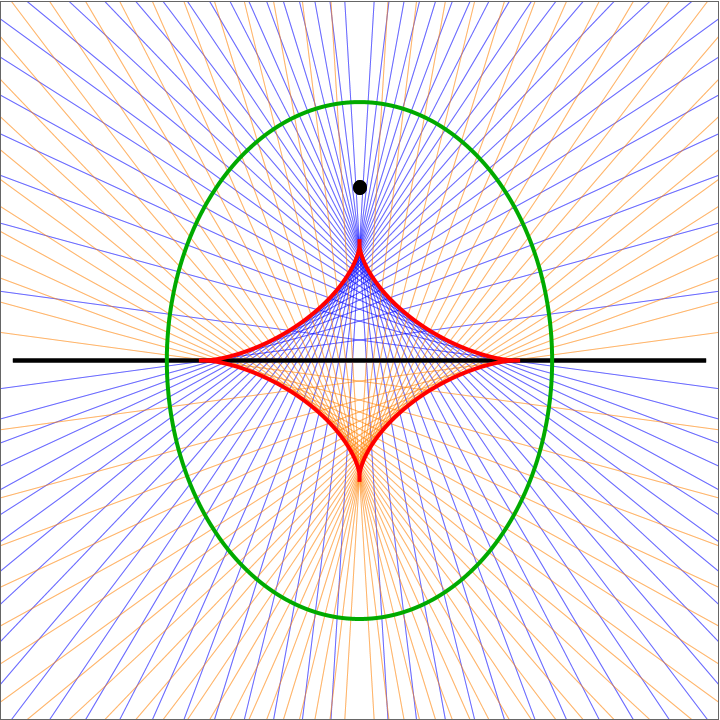}
    \caption{This plot shows refracted rays for the black radiant point and the black line. The blue lines are the refracted rays given the refraction constant $n=2/3$ and the orange lines are the refracted rays given $n=-2/3$. The red curve is the complete caustic by refraction. The green curve is the ellipse whose evolute is the complete caustic.}
    \label{fig: CausticFig line}
\end{figure}

\begin{theorem}\label{thm: Line} Let $L$ be a line. Assume that the radiant point $A$ is away from $L$, and let $B$ be the reflection of $A$ with respect to $L$. Fix a real number $n\neq 0$. If $|n|<1$, then the set of refracted rays for the refraction constants $n$ and $-n$ coincides with the set of normals to the ellipse
\begin{align}\label{eq: Cart line 1}
  \|A-M\|+\|B-M\|=\frac{\|A-B\|}{|n|}.    \end{align}
If $|n|>1$, then the set of refracted rays for the refraction constant $n$ coincides with the set of normals to the hyperbola branch
\begin{align}\label{eq: Cart line 2}
  \|A-M\|-\|B-M\|=\frac{\|A-B\|}{-n}.    \end{align}
\end{theorem}

The proof of \Cref{thm: Line} works exactly the same as for \Cref{thm: Cartesian oval}, where the circle $\mathcal C(R)$ is uniquely defined by containing $A$ and $B$, and being tangent to the normal line of $L$ at $R\in L$. To see why we get the distinct cases $|n|<1$ and $|n|>1$, fix some $R\in L$ and consider the circle $\mathcal C(R)$ containing $A,B,$ and $R$. It is clear that $A$ and $B$ are opposite in the corresponding quadrilateral precisely when $|n|<1$. That $A$ and $B$ are opposites further correspond to the signs $(\sigma_1,\sigma_2)=(1,1)$, which is reflected in \eqref{eq: Cart line 1}. 

In the following, we observe that \eqref{eq: Cart line 1} is a limit of \eqref{eq: Cart}, and note that this is also the case for \eqref{eq: Cart line 2}. We do this by fixing $A$ and a line $L$ away from $A$, and letting $B$ be the reflection of $A$ with respect to $L$. We choose a sequence of circles $\mathcal C_k$ of centers $O_k$ and radii $r_k$ that tend to the line $L$. By \eqref{eq: B def}, we get a sequence of points $B_k$. We show that $B_k$ tends to the reflection $B$ of $A$ with respect to $L$ and that \eqref{eq: Cart} tends to \eqref{eq: Cart line 1} as $k\to \infty$. The line $L$ divides $\RR^2$ into two half-planes, one that contains $A$ and one that does not. Inside the latter, let $O_k\in L_{AB}$ be any sequence of points that tends to infinity. Setting $r_k:=\|P-O_k\|$, the line $L$ can be viewed as the limit of circles $\mathcal C_k$ with centers $O_k$ and radii $r_k$. We now argue that $B_k\to B$. Up to translation, rotation and scaling, we may assume that $A=(0,0)$ and $L=\{X\in \RR^2:x_1=1\}$. Then, we have $O_k=(o_k,0)$ with $o_k\to \infty$ and $r_k=o_k-1$, and the formula \eqref{eq: B def} gives us
\begin{align}
    B_k=\begin{bmatrix} 2-\frac{1}{o_k} \\ 0       
    \end{bmatrix}\to B.
\end{align}
Finally, since $\|A-O_k\|/r_k\to 1$, we conlude that 
\begin{align}
 \|A-M\|+\sigma_1 \frac{\|A-O_k\|}{r_k}\|B_k-M\|=\sigma_2\frac{\|A-B_k\|\|A-O_k\|}{r_k|n|}   
\end{align}
tends to
\begin{align}
  \|A-M\|+\sigma_1\|B-M\|=\sigma_2\frac{\|A-B\|}{|n|}.    \end{align}
By assumption, $|n|<1$, and the only real oval among the four ovals corresponds to $(\sigma_1,\sigma_2)=(1,1)$ as stated in \Cref{ss: Cart o}.



\section{Symbolic Computation of Envelopes}\label{s: Comp} 

Let $F(y_1,y_2,t)$ be a nonzero polynomial in three variables that for fixed $t\in \RR$ defines a curve in $(y_1,y_2)\in \RR^2$. Following differential geometry arguments from \cite{bruce1992curves}, the envelope of the family of curves $V_t:=\{(y_1,y_2):F(y_1,y_2,t)=0\}$ is defined in \cite[Section 3.4]{cox1994ideals} as the solution in $(y_1,y_2)$ to the system
\begin{align}\begin{aligned}\label{eq: envelope t}
    F(y_1,y_2,t)=0,\\
    \frac{\partial F}{\partial t}(y_1,y_2,t) =0.
\end{aligned}
\end{align}
This means that we obtain the envelope by eliminating the variable $t$ from \eqref{eq: envelope t}. Seeing $F$ as a polynomial in $t$ with coefficients in $(y_1,y_2)$, the envelope is equivalently given by the determinant of the Sylvester matrix of $F$ and $\partial F/\partial t$. In simple examples, we can use this definition to calculate the envelope by hand.

\begin{example} Consider the parabola parametrized by $(t^2,t)$. We parametrize the normal at $(t^2,t)$ by $(t^2,t)+\lambda (-1,2t)$, $\lambda\in \RR$. Therefore, we define   
\begin{align}
    F(y_1,y_2,t):=\det \begin{bmatrix}
        y_1-t^2 & -1\\ y_2-t & 2t
    \end{bmatrix}=2ty_1+y_2-2t^3-t.
\end{align}
Then, \eqref{eq: envelope t} becomes
\begin{align}
   2ty_1+y_2-2t^3 -t&=0,\label{eq: first}\\
   2y_1-6t^2-1&=0.\label{eq: sec}
\end{align}
From \eqref{eq: sec}, we get $y_1=3t^2+1/2$ and combining this with \eqref{eq: sec}, we get $y_2=-4t^3$. Eliminating $t$ by hand yields the envelope $2(2y_1-1)^3=27y_2^2$. 
\hfill$\diamondsuit\,$
\end{example}

In the setting of \Cref{s: Circles}, let $\mathcal C$ is a circle with center $O$ and radius $r$, parametrized rationally by $R(t)$. Then, $N_t=R(t)-O$ is the normal to the circle at $t$. We rewrite \eqref{eq: real point} as
\begin{align}\begin{aligned}\label{eq: real point FF}
 \|Y-R(t)\|^2 \Big(\big(A-R(t)\big)\cdot \iota(N_t)\Big)^2-n^2\|A-R(t)\|^2\Big(\big(Y-R(t)\big)\cdot \iota(N_t)\Big)^2=0, 
\end{aligned}
\end{align}
and note that it is a rational expression in $t$. We define the following polynomial in the polynomial ring $\CC[y_1,y_2,t,r,n]$, 
\begin{align}\label{eq: real point F}
    F_{r,n}(y_1,y_2,t):=\textnormal{ numerator of } \eqref{eq: real point FF}.
\end{align}
We may without loss of generality fix $A=(0,0)$ and $O=(1,0)$; refracted rays are naturally equivariant with respect to translation, rotation, and scaling. Viewing $F_{n,r}$ as a polynomial in $t$ with coefficients that are polynomial in $y_1,y_2,r$ and $n$, this allows us to calculate the determinant of the Sylvester matrix using \texttt{Macaulay2} via the code below. This took 2 hours, 19 minutes and 26 seconds to run on a 13th Gen Intel(R) Core(TM) i9-13900H CPU running at 2.60 GHz.
\begin{lstlisting}[language=Macaulay2]
R = QQ[y_1,y_2,t,r,n]

iota = (V) -> matrix{{-V_(1,0)},{V_(0,0)}}

A = matrix{{0},{0}}
Circ = matrix{{1},{0}}

R = matrix{{r*(2*t)/(1+t^2)},{r*(t^2-1)/(1+t^2)}}+Circ 
Y = matrix{{y_1},{y_2}}

diffAR = transpose(A-R)*(A-R)
diffYR = transpose(Y-R)*(Y-R)
normal = R-O
dotAxnormal = transpose(A-R)*iota(normal)
dotyxnormal = transpose(Y-R)*iota(normal)

F = diffYR*dotARnormal^2-n^2*diffAR*dotYRnormal^2
Fpoly = numerator(F_(0)_(0))

S = QQ[y_1,y_2,r,n][t]
Fpoly = sub(Fpoly,S)

factF = (factor(Fpoly))#1#0
SMAT = transpose(sylvesterMatrix(factF,diff(t,factF),t))
Final = det SMAT
\end{lstlisting}
The computed  polynomial consists of the factors:
\begin{align}\begin{aligned}
    &n^4, \quad (r-1)^2,\quad (r+1)^2, \quad y^2,\quad \big((x-1)^2+y^2-r^2\big)^2,\\
    &(x-1)^2(r^2n^2+n^2-1)-(y-r)^2,
\end{aligned}
\end{align}
and a polynomial 
\begin{align}\label{eq: long eq}
x^{12}r^{12}n^{12}+6x^{10}y^2r^{12}n^{12}+15x^8y^4r^{12}n^{12}+\cdots -6n^2-12x+1
\end{align}
of 1765 terms that is degree 12 in the four variables $x,y,r$ and $n$. 

As mentioned in the introduction, the evolute of a plane curve coincides with its ED-discriminant. We may compute the evolute of the Cartesian ovals from \Cref{ex: main} and \Cref{fig: CausticFig} as in \cite[Section 7]{draisma2016euclidean}. The result is a degree-12 equation of 49 terms: 
\begin{align}16384000x^{12} + 351768768x^{10}y^2  + \cdots- 167215104x + 11943936.
\end{align}
Specifying $r=1/3$ and $n=1/2$, this polynomial equals \eqref{eq: long eq}; the factor \eqref{eq: long eq} corresponds to complete caustics by refraction. Keeping $r$ and $n$ unknown in the ED-discriminant code resulted in a program that did not terminate in reasonable time.


In the example below, we perform similar computations for other conic sections and provide illustrations in \Cref{fig: Evolutes-Conic Sections}. This is possible since all conic sections may be parametrized in a similar way to the circle. However, we were not able to compute a general formula for the complete caustic, as we did for the circle. This is because too many variables are required to specify parabolas, ellipses, and hyperbolas for our code to terminate within reasonable time.

\begin{example}\label{ex: conic sections} Fix the radiant point $A=(0,0)$ and consider the parabola, ellipse and hyperbola 
\begin{align}
-x_1^2-x_2=1,\quad x_1^2+4(x_2+1)^2-1, \quad\textnormal{ and }\quad  x_1^2-4(x_2+1)^2=1,
\end{align}
respectively. In each case, with $n=2/3$, we calculate the red caustics shown in \Cref{fig: Evolutes-Conic Sections} using \eqref{eq: envelope t}. Their degrees are 18, 24 and 24, respectively.
\end{example}

\begin{figure}
\begin{minipage}[h]{0.47\linewidth}
\begin{center}
\includegraphics[width=0.7\linewidth]{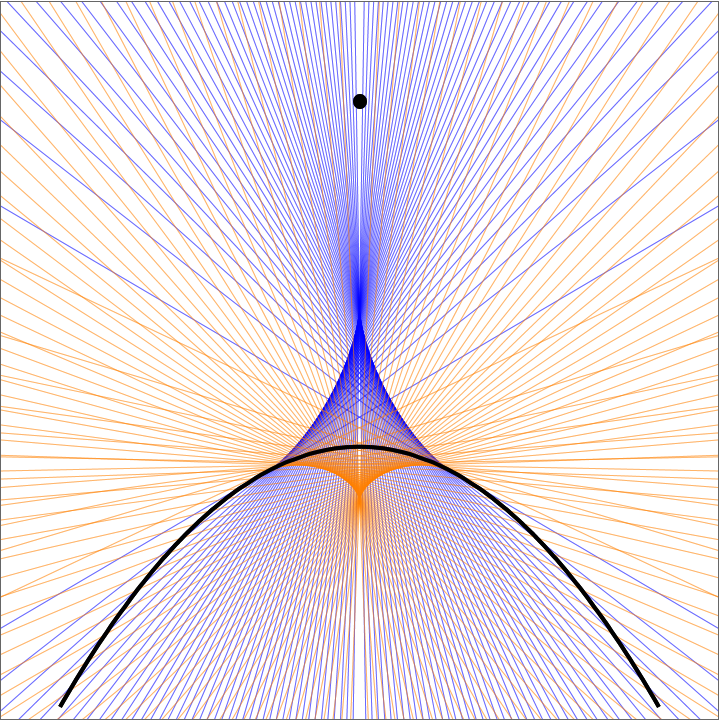} 
\label{qwe1}
\end{center}
\end{minipage}
\hspace{-1.75cm}
\begin{minipage}[h]{0.47\linewidth}
\begin{center}
\includegraphics[width=0.7\linewidth]{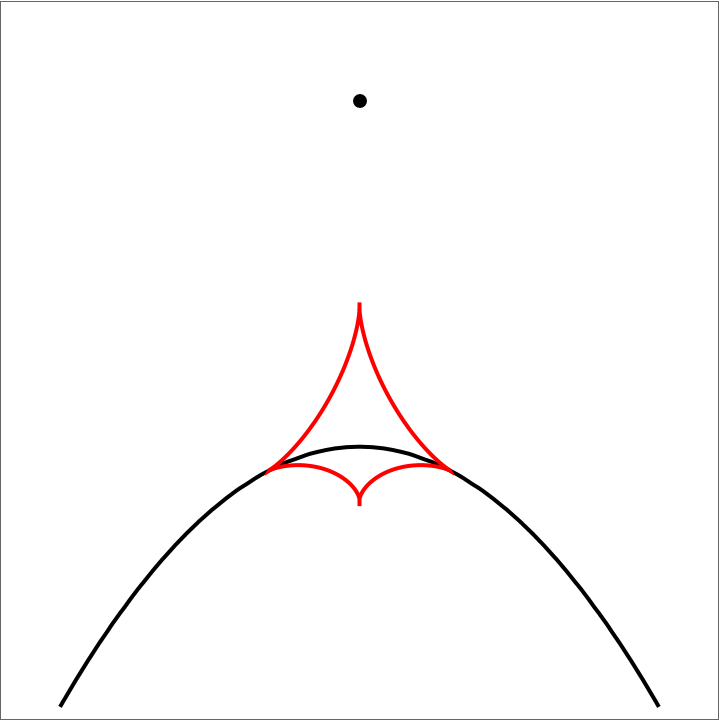} 
\label{qwe2}
\end{center}
\end{minipage}
\vfill
\vspace{0.25cm}
\begin{minipage}[h]{0.47\linewidth}
\begin{center}
\includegraphics[width=0.7\linewidth]{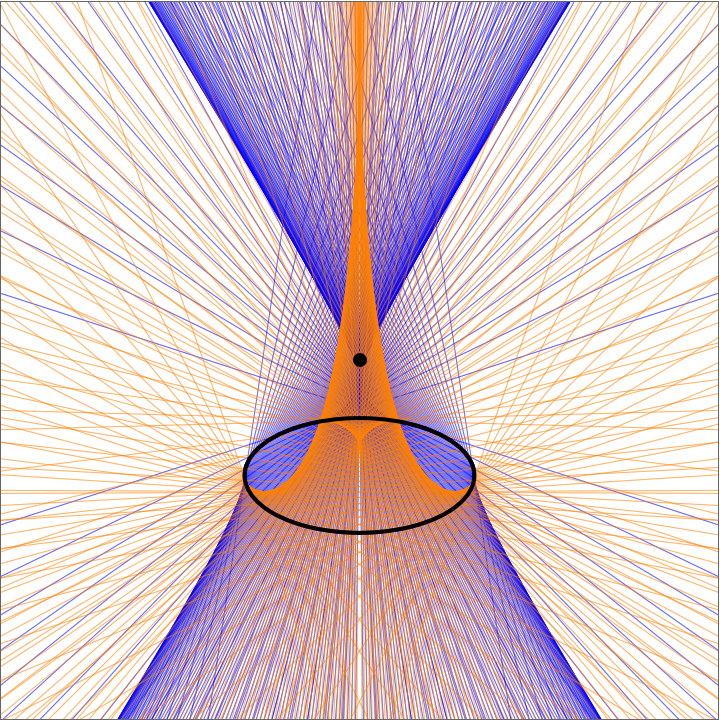} 
\label{qwe3}
\end{center} 
\end{minipage}
\hspace{-1.75cm}
\begin{minipage}[h]{0.47\linewidth}
\begin{center}
\includegraphics[width=0.7\linewidth,angle=-90]{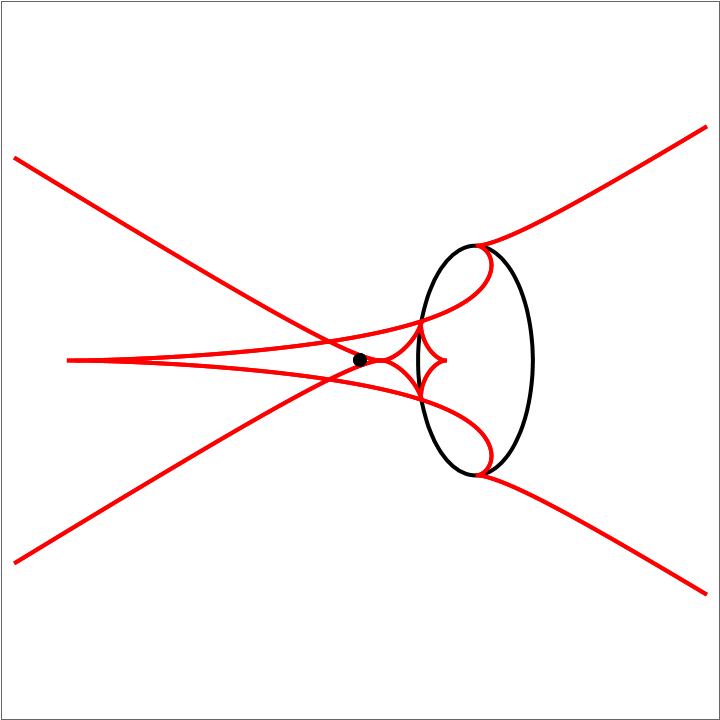} 
\label{qwe4}
\end{center}
\end{minipage}

\vfill
\vspace{0.25cm}
\begin{minipage}[h]{0.47\linewidth}
\begin{center}
\includegraphics[width=0.7\linewidth]{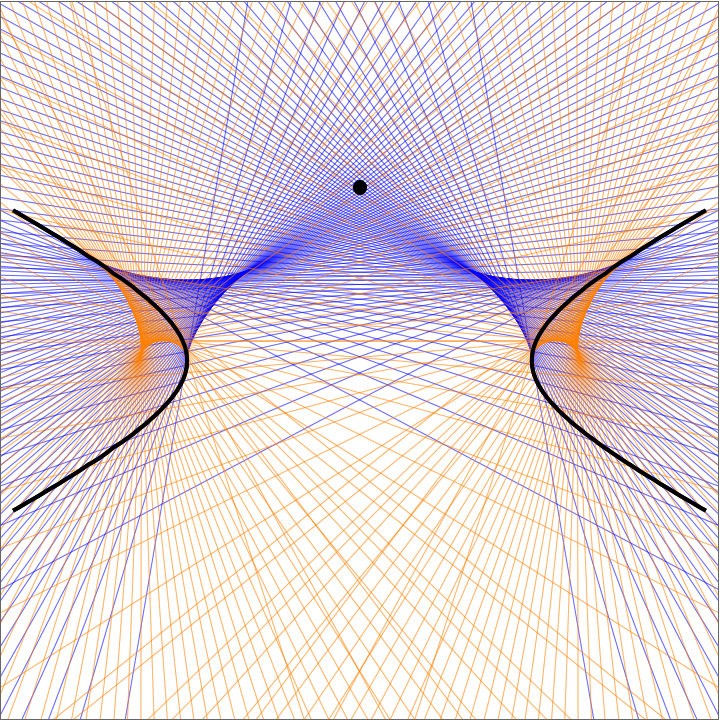} 
\label{qwe7}
\end{center}
\end{minipage}
\hspace{-1.75cm}
\begin{minipage}[h]{0.47\linewidth}
\begin{center}
\includegraphics[width=0.7\linewidth]{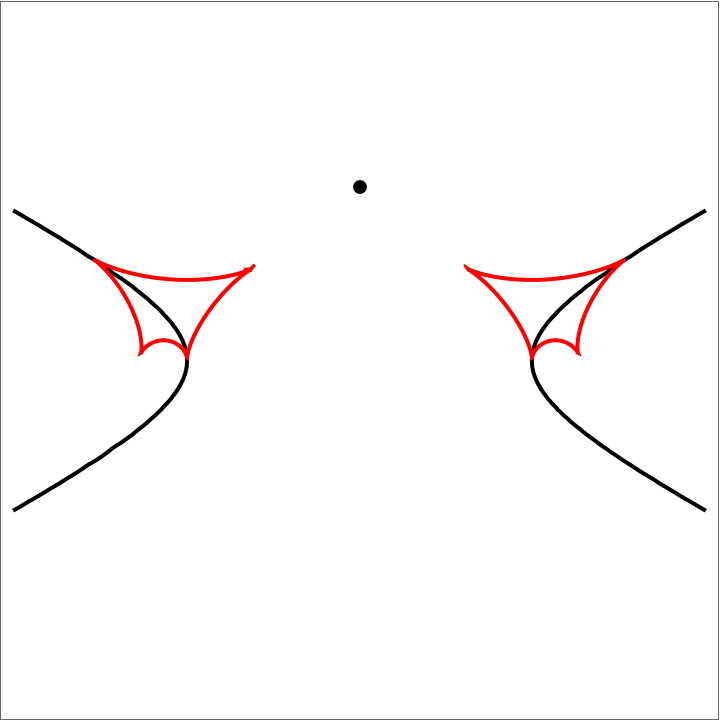} 
\label{qwe8}
\end{center}
\end{minipage}
    \caption{Complete caustics by refraction of a parabola, ellipse, and hyperbola across three rows. Their equations are given in \Cref{ex: conic sections}. Radiant points and conic sections are shown in black. Left column: Refracted rays for $n=2/3$ (blue) and $n=-2/3$ (orange). Right column: Caustic curves (red). }\label{fig: Evolutes-Conic Sections}
\end{figure}

We leave it as an open problem to (if they exist) find curves whose evolutes are the complete caustics shown in \Cref{fig: Evolutes-Conic Sections}, and to describe such curves given arbitrary parabolas, ellipses and hyperbolas.




\bibliographystyle{alpha}
\bibliography{VisionBib}



\end{document}